\documentclass[11pt, reqno]{amsart}
\usepackage{amsmath, amsthm, amscd, amsfonts, amssymb, mathrsfs, graphicx, color}
\usepackage[bookmarksnumbered, colorlinks, plainpages]{hyperref}
\hypersetup{colorlinks=true,linkcolor=red, anchorcolor=green, citecolor=cyan, urlcolor=red, filecolor=magenta, pdftoolbar=true}

\theoremstyle{plain}
\newtheorem{theorem}{Theorem}[section]
\newtheorem{corollary}[theorem]{Corollary}
\newtheorem{lemma}[theorem]{Lemma}
\newtheorem{proposition}[theorem]{Proposition}

\theoremstyle{definition}
\newtheorem{definition}{Definition}[section]

\theoremstyle{remark}
\newtheorem{remark}{Remark}[section]

\numberwithin{equation}{section}

\numberwithin{table}{section}

\numberwithin{figure}{section} 
\begin{document}
\title[Means of Accretive Operators]
{Relative Entropy and Tsallis Entropy of two Accretive Operators}
\author[M. Ra\"{\i}ssouli, M.S. Moslehian, S. Furuichi]{M. Ra\"{\i}ssouli$^{1,2}$, M. S. Moslehian$^3$, and S. Furuichi$^4$}
\address{$^1$ Department of Mathematics, Science Faculty, Taibah University,
Al Madinah Al Munawwarah, P.O.Box 30097, Zip Code 41477, Saudi Arabia.}
\address{$^2$ Department of Mathematics, Faculty of Science, Moulay Ismail University, Meknes, Morocco.}
\email{raissouli.mustapha@gmail.com}

\address{$^3$ Department of Pure Mathematics, P.O. Box 1159, Ferdowsi University of Mashhad, Mashhad 91775, Iran.}
\email{moslehian@um.ac.ir}

\address{$^4$ Department of Information Science,
College of Humanities and Sciences, Nihon University,
3-25-40, Sakurajyousui, Setagaya-ku, Tokyo, 156-8550, Japan.}
\email{furuichi@chs.nihon-u.ac.jp}


\begin{abstract}
Let $A$ and $B$ be two accretive operators. We first introduce the weighted geometric mean of $A$ and $B$ together with some related properties. Afterwards, we define the relative entropy as well as the Tsallis entropy of $A$ and $B$. The present definitions and their related results extend those already introduced in the literature for positive invertible operators.
\end{abstract}

\maketitle

\section{\bf Introduction}

Let $\big(H,\langle.,.\rangle\big)$ be a complex Hilbert space and let ${\mathcal B}(H)$ be the ${\mathbb C}^*$-algebra of bounded linear operators acting on $H$. Every $A\in{\mathcal B}(H)$ can be written in the following form
\begin{equation}\label{eqCD}
A=\Re A+i\Im A,\;\; \mbox{with}\;\; \Re A=\frac{A+A^*}{2}\;\; \mbox{and}\;\; \Im A=\frac{A-A^*}{2i}.
\end{equation}
This is known in the literature as the so-called Cartesian decomposition of $A$, where the operators $\Re A$ and $\Im A$ are the real and imaginary parts of $A$, respectively. As usual, if $A$ is self-adjoint (i.e. $A^*=A$ ) we say that $A$ is positive if $\langle Ax,x\rangle\geq0$ for all $x\in H$ and, $A$ is strictly positive if $A$ is positive and invertible. For $A,B\in{\mathcal B}(H)$ self-adjoint, we write $A\leq B$ or $B\geq A$ for meaning that $B-A$ is positive.

If $A,B\in{\mathcal B}(H)$ are strictly positive and $\lambda\in(0,1)$ is a real number, then the following
\begin{equation}\label{eq11}
A\nabla_{\lambda}B:=(1-\lambda)A+\lambda B,\, A!_{\lambda}B:=\Big((1-\lambda)A^{-1}+\lambda B^{-1}\Big)^{-1},\, A\sharp_{\lambda} B:=A^{1/2}\Big(A^{-1/2}BA^{-1/2}\Big)^{\lambda}A^{1/2}
\end{equation}
are known, in the literature, as the $\lambda$-weighted arithmetic, $\lambda$-weighted harmonic and $\lambda$-weighted geometric operator means of $A$ and $B$, respectively. If $\lambda=1/2$, they are simply denoted by $A\nabla B,\; A!B$ and $A\sharp B$, respectively. The following inequalities are well-known in the literature:
\begin{equation}\label{eqAGH}
A!_{\lambda}B\leq A\sharp_{\lambda}B\leq A\nabla_{\lambda}B.
\end{equation}

For more details about the previous operator means, as well as some other weighted and generalized operator means, we refer the interested reader to the recent paper \cite{PSMA} and the related references cited therein. For refined and reversed inequalities of (\ref{eqAGH}) one can consult \cite{HKKLP} and \cite{KKLP} for more information.

Now, let $A\in{\mathcal B}(H)$ be as in (\ref{eqCD}). We say that $A$ is accretive if its real part $\Re A$ is strictly positive. If $A,B\in{\mathcal B}(H)$ are accretive then so are $A^{-1}$ and $B^{-1}$. Further, it is easy to see that the set of all accretive operators acting on $H$ is a convex cone of ${\mathcal B}(H)$. Consequently, $A\nabla_{\lambda}B$ and $A!_{\lambda}B$ can be defined by the same formulas as previous whenever $A,B\in{\mathcal B}(H)$ are accretive. Clearly, the relationships
$A\nabla_{\lambda}B=B\nabla_{1-\lambda}A,\;\; A!_{\lambda}B=B!_{1-\lambda}A,\;\; A!_{\lambda}B=\big(A^{-1}\nabla_{\lambda}B^{-1}\big)^{-1}$
are also valid for any accretive $A,B\in{\mathcal B}(H)$ and $\lambda\in(0,1)$.

However, $A\sharp_{\lambda}B$ can not be defined by the same formula (\ref{eq11}) when $A,B\in{\mathcal B}(H)$ are accretive, by virtue of the presence of non-integer exponents for operators in (\ref{eq11}). For the particular case $\lambda=1/2$, Drury \cite{D} defined $A\sharp B$ via the following formula (where we continue to use the same notation)
\begin{equation}\label{eq12}
A\sharp B=\left(\frac{2}{\pi}\int_0^{\infty}\big(tA+t^{-1}B\big)^{-1}\frac{dt}{t}\right)^{-1}.
\end{equation}
It is proved in \cite{D} that $A\sharp B=B\sharp A$ and $A\sharp B=\big(A^{-1}\sharp B^{-1}\big)^{-1}$ for any accretive $A,B\in{\mathcal B}(H)$. It follows that (\ref{eq12}) is equivalent to:
\begin{equation}\label{eq125}
A\sharp B=\frac{2}{\pi}\int_0^{\infty}\big(tA^{-1}+t^{-1}B^{-1}\big)^{-1}\frac{dt}{t}=\frac{2}{\pi}\int_0^{\infty}A\big(tB+t^{-1}A\big)^{-1}B\frac{dt}{t}.
\end{equation}

In this paper we will define $A\sharp_{\lambda}B$ when the operators $A,B\in{\mathcal B}(H)$ are accretive. Some related operator inequalities are investigated. We also introduce the relative entropy and the Tsallis entropy for this class of operators.

\section{\bf Weighted Geometric Mean}

We start this section by stating the following definition which is the main tool for the present approach.

\begin{definition}
Let $A,B\in{\mathcal B}(H)$ be two accretive operators and let $\lambda\in(0,1)$. The $\lambda$-weighted geometric mean of $A$ and $B$ is defined by
\begin{equation}\label{eq21}
A\sharp_{\lambda}B:=\frac{\sin(\lambda\pi)}{\pi}\int_0^{\infty}t^{\lambda-1}\Big(A^{-1}+tB^{-1}\Big)^{-1}\;dt
=\frac{\sin(\lambda\pi)}{\pi}\int_0^{\infty}t^{\lambda-1}A\big(B+tA\big)^{-1}B\;dt.
\end{equation}
\end{definition}

In the aim to justify our previous definition we first state the following.

\begin{proposition}\label{pr21}
The following assertions are true:\\
(i) If $A,B\in{\mathcal B}(H)$ are strictly positive then (\ref{eq21}) coincides with (\ref{eq11}).\\
(ii) If $\lambda=1/2$ then (\ref{eq21}) coincides with (\ref{eq12}).
\end{proposition}
\begin{proof}
(i) Assume that $A,B\in{\mathcal B}(H)$ are strictly positive. From (\ref{eq21}) it is easy to see that
$$A\sharp_{\lambda}B=A^{1/2}\left(\frac{\sin(\lambda\pi)}{\pi}\int_0^{\infty}t^{\lambda-1}\Big(I+tA^{1/2}B^{-1}A^{1/2}\Big)^{-1}\;dt\right)A^{1/2},$$
where $I$ denotes the identity operator on $H$. Since $A^{1/2}B^{-1}A^{1/2}$ is self-adjoint strictly positive then it is sufficient, by virtue of (\ref{eq11}), to show that the following equality
$$a^{-\lambda}=\frac{\sin(\lambda\pi)}{\pi}\int_0^{\infty}t^{\lambda-1}\big(1+ta\big)^{-1}dt$$
holds for all real number $a>0$. If we make the change of variables $u=(1+ta)^{-1}$, the previous real integral becomes after simple manipulations (here the notations $B$ and $\Gamma$ refer to the standard beta and gamma functions)
$$\frac{\sin(\lambda\pi)}{a^{\lambda}\pi}\int_0^1(1-u)^{\lambda-1}u^{-\lambda}du=\frac{1}{a^{\lambda}}
\frac{\sin(\lambda\pi)}{\pi}B\big(\lambda,1-\lambda\big)=\frac{1}{a^{\lambda}}
\frac{\sin(\lambda\pi)}{\pi}\Gamma(\lambda)\Gamma(1-\lambda)=\frac{1}{a^{\lambda}}.$$
The proof of (i) is finished.\\
(ii) Let $A,B\in{\mathcal B}(H)$ be accretive. If $\lambda=1/2$ then (\ref{eq21}) yields
$$A\sharp B=\frac{1}{\pi}\int_0^{\infty}\big(A^{-1}+tB^{-1}\big)^{-1}\frac{dt}{\sqrt{t}},$$
which, with the change of variables $u=\sqrt{t}$, becomes (after simple computation)
$$A\sharp B=\frac{2}{\pi}\int_0^{\infty}\big(A^{-1}+u^2B^{-1}\big)^{-1}du=\frac{2}{\pi}\int_0^{\infty}\big(u^{-1}A^{-1}+uB^{-1}\big)^{-1}\frac{du}{u}.$$
This, with (\ref{eq125}) and the fact that $A\sharp B=B\sharp A$, yields the desired result. The proof of the proposition is completed.
\end{proof}

From a functional point of view, we are allowing to state another equivalent form of (\ref{eq21}) which seems to be more convenient for our aim in the sequel.

\begin{lemma}\label{lem21}
For any accretive $A,B\in{\mathcal B}(H)$ and $\lambda\in(0,1)$, there holds
\begin{equation}\label{eq22}
A\sharp_{\lambda}B=\frac{\sin(\lambda\pi)}{\pi}\int_0^{1}\frac{t^{\lambda-1}}{(1-t)^{\lambda}}\big(A!_{t}B\big)dt.
\end{equation}
\end{lemma}
\begin{proof}
If in (\ref{eq21}) we make the change of variables $t=u/(1-u),\; u\in[0,1)$, we obtain the desired result after simple topics of real integration. Detail is simple and therefore omitted here.
\end{proof}

Using the previous lemma, it is not hard to verify that the following formula
$$A\sharp_{\lambda}B=B\sharp_{1-\lambda}A$$
persists for any accretive $A,B\in{\mathcal B}(H)$ and $\lambda\in(0,1)$. Moreover, it is clear that $\Re\big(A\nabla_{\lambda}B\big)=(\Re A)\nabla_{\lambda}(\Re B)$. About $A!_{\lambda}B$ we state the following lemma which will also be needed in the sequel.

\begin{lemma}\label{lem22}
For any accretive $A,B\in{\mathcal B}(H)$ and $\lambda\in(0,1)$, it holds that
\begin{equation}\label{eq23}
\Re\big(A!_{\lambda}B\big)\geq\big(\Re A\big)!_{\lambda}\big(\Re B\big).
\end{equation}
\end{lemma}
\begin{proof}
Let $f(A)=\big(\Re(A^{-1})\big)^{-1}$ be defined on the convex cone of accretive operators $A\in{\mathcal B}(H)$. In \cite{M}, Mathias proved that $f$ is operator convex, i.e.
$$f\Big((1-\lambda)A+\lambda B\Big)\leq(1-\lambda)f(A)+\lambda f(B).$$
This means that
$$\Big(\Re\big((1-\lambda)A+\lambda B\big)^{-1}\Big)^{-1}\leq(1-\lambda)\big(\Re(A^{-1})\big)^{-1}+\lambda\big(\Re(B^{-1})\big)^{-1}.$$
Replacing in this latter inequality $A$ and $B$ by the accretive operators $A^{-1}$ and $B^{-1}$, respectively, and using the fact that the map $X\longmapsto X^{-1}$ is operator monotone increasing for $X\in{\mathcal B}(H)$ strictly positive, we then deduce (\ref{eq23}).
\end{proof}

We now are in a position to state our first main result (which extends Theorem 1.1 of \cite{LS}).

\begin{theorem}
Let $A,B\in{\mathcal B}(H)$ be accretive and $\lambda\in(0,1)$. Then
\begin{equation}\label{eq24}
\Re\Big(A\sharp_{\lambda}B\Big)\geq(\Re A)\sharp_{\lambda}(\Re B).
\end{equation}
\end{theorem}
\begin{proof}
By (\ref{eq22}) with (\ref{eq23}) we can write
$$\Re\Big(A\sharp_{\lambda}B\Big)=\frac{\sin(\lambda\pi)}{\pi}\int_0^{1}\frac{t^{\lambda-1}}{(1-t)^{\lambda}}\Re\big(A!_{t}B\big)dt
\geq\frac{\sin(\lambda\pi)}{\pi}\int_0^{1}\frac{t^{\lambda-1}}{(1-t)^{\lambda}}\big(\Re A\big)!_{t}\big(\Re B\big)dt,$$
which, when combined with Proposition \ref{pr21}, implies the desired result.
\end{proof}

\section{\bf Relative/Tsallis Operator Entropy}

Let $A,B\in{\mathcal B}(H)$ be strictly positive and $\lambda\in(0,1)$. The relative operator entropy ${\mathcal S}(A|B)$ and the Tsallis relative operator entropy
${\mathcal T}_{\lambda}(A|B)$ were defined by
\begin{equation}\label{eq31}
{\mathcal S}(A|B):=A^{1/2}\log\Big(A^{-1/2}BA^{-1/2}\Big)A^{1/2}.
\end{equation}
\begin{equation}\label{eq32}
{\mathcal T}_{\lambda}(A|B):=\frac{A\sharp_{\lambda}B-A}{\lambda},
\end{equation}
see \cite{FK, FS, F} for instance. The Tsallis relative operator entropy is a parametric extension in the sense that
\begin{equation}\label{eq312_limit}
\lim_{\lambda \to 0} {\mathcal T}_{\lambda}(A|B) = {\mathcal S}(A|B).
\end{equation}
For more details about these operator entropies, we refer the reader to \cite{F2} and \cite{MMM} and the related references cited therein.

Our aim in this section is to extend ${\mathcal S}(A|B)$ and ${\mathcal T}_{\lambda}(A|B)$ for accretive $A,B\in{\mathcal B}(H)$. Following the previous study we suggest that ${\mathcal T}_{\lambda}(A|B)$ can be defined by the same formula (\ref{eq32}) whenever $A,B\in{\mathcal B}(H)$ are accretive and so $A\sharp_{\lambda}B$ is given by (\ref{eq22}). Precisely, we have

\begin{definition}
Let $A,B\in{\mathcal B}(H)$ be accretive and let $\lambda \in (0,1)$. The Tsallis relative operator entropy of $A$ and $B$ is defined by
\begin{equation}\label{eq33_Tsallis}
{\mathcal T_{\lambda}}(A|B)= \frac{\sin \lambda \pi}{\lambda \pi} \int_0^1 \left( \frac{t}{1-t} \right)^{\lambda}\left( \frac{A !_t B -A}{t} \right) dt.
\end{equation}
\end{definition}

This, with (\ref{eq32}) and (\ref{eq24}), immediately yields
$$\Re\Big({\mathcal T}_{\lambda}(A|B)\Big)\geq{\mathcal T}_{\lambda}\big(\Re A|\Re B\big)$$
for any accretive $A,B\in{\mathcal B}(H)$ and $\lambda\in(0,1)$.

In view of (\ref{eq33_Tsallis}), extension of ${\mathcal S}(A|B)$ can be introduced via the following definition (where we always conserve the same notation, for the sake of simplicity).

\begin{definition}
Let $A,B\in{\mathcal B}(H)$ be accretive. The relative operator entropy
 of $A$ and $B$ is defined by
\begin{equation}\label{eq33}
{\mathcal S}(A|B)=\int_0^1\frac{A!_tB-A}{t}dt.
\end{equation}
\end{definition}

The following proposition gives a justification as regards the previous definition.

\begin{proposition}\label{pr31}
If $A,B\in{\mathcal B}(H)$ are strictly positive then (\ref{eq33}) coincides with (\ref{eq31}).
\end{proposition}
\begin{proof}
Assume that $A,B\in{\mathcal B}(H)$ are strictly positive. By (\ref{eq33}), with the definition of $A!_{t}B$, it is easy to see that
$${\mathcal S}(A|B)=A^{1/2}\left(\int_0^{1}\frac{\Big((1-t)I+tA^{1/2}B^{-1}A^{1/2}\Big)^{-1}-I}{t}dt\right)A^{1/2}.$$
By similar arguments as those for the proof of Proposition \ref{pr21}, it is sufficient to show that
$$\log a=\int_0^1\frac{\big(1-t+ta^{-1}\big)^{-1}-1}{t}dt$$
is valid for any $a>0$. This follows from a simple computation of this latter real integral, so completing the proof.
\end{proof}

\begin{theorem}
Let $A,B\in{\mathcal B}(H)$ be accretive. Then
\begin{equation}\label{eq34}
\Re\Big({\mathcal S}(A|B)\Big)\geq{\mathcal S}\big(\Re A|\Re B\big).
\end{equation}
\end{theorem}
\begin{proof}
By (\ref{eq33}) with Lemma \ref{lem22} we have
$$\Re\big({\mathcal S}(A|B)\big)=\int_0^1\frac{\Re(A!_tB)-\Re A}{t}dt\geq\int_0^1\frac{\big(\Re A\big)!_t\big(\Re B\Big)-\Re A}{t}dt.$$
This, with Proposition \ref{pr31}, immediately yields (\ref{eq34}).
\end{proof}

\begin{proposition}\label{pr31_Tsallis}
If $A,B\in{\mathcal B}(H)$ are strictly positive then (\ref{eq33_Tsallis}) coincides with (\ref{eq32}).
\end{proposition}
\begin{proof}
Putting $1-t = \frac{1}{s+1}$, (\ref{eq33_Tsallis}) is calculated as
$${\mathcal T_{\lambda}}(A|B)
= \frac{\sin \lambda \pi}{\lambda \pi} \int_0^{\infty} s^{\lambda -1} \left\{ \left(A^{-1}+s B^{-1} \right)^{-1} -(1+s)^{-1}A\right\} ds.
$$
For $A,B >0$, we have
$$
\frac{\sin \lambda \pi}{\pi} \int_0^{\infty} s^{\lambda-1} \left(A^{-1}+s B^{-1} \right)^{-1} ds =A \sharp_{\lambda}B
$$
and
$$
\frac{\sin \lambda \pi}{\pi} \int_0^{\infty} s^{\lambda-1} \left(1+s \right)^{-1} ds =1,
$$
which imply the assertion.

\end{proof}

We note that (ii) of Proposition \ref{pr31_Tsallis} is a generalization of (\ref{eq312_limit}).
We end this section by stating the following remark.

\begin{remark}
Analog of (\ref{eqAGH}), for accretive $A,B\in{\mathcal B}(H)$, does not persist, i.e.
$$\Re\Big(A!_{\lambda}B\Big)\leq\Re\Big(A\sharp_{\lambda}B\Big)\leq\Re\Big(A\nabla_{\lambda}B\Big)$$
fail for some accretive $A,B\in{\mathcal B}(H)$. For $\lambda=1/2$, this was pointed out in \cite{LS} and the same arguments may be used for general $\lambda\in(0,1)$.
\end{remark}

However, the following remark worth to be mentioned.

\begin{remark}
In \cite{LIN} (see Section 3, Theorem 3), M. Lin presented an extension of the geometric mean-arithmetic mean inequality
$A\sharp B\leq A\nabla B$
from positive matrices to accretive matrices (called there, sector matrices). By similar arguments, we can obtain an analogue inequality between the $\lambda$-weighted geometric mean $A\sharp_{\lambda}B$ and the $\lambda$-weighted arithmetic mean $A\nabla_{\lambda}B$, when $A$ and $B$ are sector matrices. We omit the details about this latter point to the reader.
\end{remark}

\section{\bf More about $A\sharp_{\lambda}B$}

We preserve the same notation as previous. The operator mean $A\sharp_{\lambda}B$ enjoys more other properties which we will discuss in this section. For any real numbers $\alpha,\beta>0$ we set $\alpha\sharp_{\lambda}\beta=\alpha^{1-\lambda}\beta^{\lambda}$ the real $\lambda$-weighted geometric mean of $\alpha$ and $\beta$.

Now, the following proposition may be stated.

\begin{proposition}\label{pr41}
For any accretive $A,B\in{\mathcal B}(H)$ and $\lambda\in(0,1)$ the following equality
\begin{equation}\label{eq41}
(\alpha A)\sharp_{\lambda}(\beta B)=(\alpha\sharp_{\lambda}\beta)\big(A\sharp_{\lambda}B\big)
\end{equation}
holds for every real numbers $\alpha,\beta>0$.
\end{proposition}
\begin{proof}
Since $A\sharp_{\lambda}B=B\sharp_{1-\lambda}A$ it is then sufficient to prove that
$(\alpha A)\sharp_{\lambda}B=\alpha^{1-\lambda}(A\sharp_{\lambda}B)$. By equation (\ref{eq21}), we have
$$(\alpha A)\sharp_{\lambda}B=\frac{\sin(\lambda\pi)}{\pi}\int_0^{\infty}t^{\lambda-1}
\Big(\frac{A^{-1}}{\alpha}+tB^{-1}\Big)^{-1}dt=\alpha\frac{\sin(\lambda\pi)}{\pi}\int_0^{\infty}t^{\lambda-1}
\big(A^{-1}+\alpha tB^{-1}\Big)^{-1}dt.$$
If we make the change of variables $u=\alpha t$, and we use again (\ref{eq21}), we immediately obtain the desired equality after simple manipulations.
\end{proof}

We now state the following result which is also of interest.

\begin{theorem}\label{th41}
Let $A,B\in{\mathcal B}(H)$ be accretive and $\lambda\in(0,1)$. Then the following inequality
\begin{equation}\label{eq42}
\sum_{k=1}^n\langle\Big(\Re\big(A\sharp_{\lambda}B\big)\Big)^{-1}x_k,x_k\rangle\leq\left(\sum_{k=1}^n\langle\big(\Re A\big)^{-1}x_k,x_k\rangle\right)
\sharp_{\lambda}
\left(\sum_{k=1}^n\langle\big(\Re B\big)^{-1}x_k,x_k\rangle\right)
\end{equation}
holds true, for any family of vectors $(x_k)_{k=1}^n\in H$.
\end{theorem}
\begin{proof}
By (\ref{eq24}), with the left-side of (\ref{eqAGH}), we have
$$\Re\big(A\sharp_{\lambda}B\big)\geq(\Re A)\sharp_{\lambda}(\Re B)\geq(\Re A)!_{\lambda}(\Re B),$$
from which we deduce
$$\Big(\Re\big(A\sharp_{\lambda}B\big)\Big)^{-1}\leq(1-\lambda)(\Re A)^{-1}+\lambda(\Re B)^{-1}.$$
Replacing in this latter inequality $A$ by $tA$, with $t>0$ real number, and using Proposition \ref{pr41}, we obtain (after a simple manipulation)
$$t^{\lambda}\Big(\Re\big(A\sharp_{\lambda}B\big)\Big)^{-1}\leq(1-\lambda)(\Re A)^{-1}+t\lambda(\Re B)^{-1}.$$
This means that, for any $x\in H$ and $t>0$, we have
\begin{equation*}
t^{\lambda}\langle\Big(\Re\big(A\sharp_{\lambda}B\big)\Big)^{-1}x,x\rangle\leq
(1-\lambda)\langle(\Re A)^{-1}x,x\rangle+t\lambda\langle(\Re B)^{-1}x,x\rangle,
\end{equation*}
and so
\begin{equation}\label{eq43}
t^{\lambda}\sum_{k=1}^n\langle\Big(\Re\big(A\sharp_{\lambda}B\big)\Big)^{-1}x_k,x_k\rangle\leq
(1-\lambda)\sum_{k=1}^n\langle(\Re A)^{-1}x_k,x_k\rangle+t\lambda\sum_{k=1}^n\langle(\Re B)^{-1}x_k,x_k\rangle,
\end{equation}
holds for any $(x_k)_{k=1}^n\in H$ and $t>0$. If $x_k=0$ for each $k=1,2,...,n$ then (\ref{eq42}) is an equality. Assume that $x_k\neq0$ for some $k=1,2,...,n$. If we take
$$t=\left(\frac{\sum_{k=1}^n\langle\Big(\Re\big(A\sharp_{\lambda}B\big)\Big)^{-1}x_k,x_k\rangle}
{\sum_{k=1}^n\langle(\Re B)^{-1}x_k,x_k\rangle}\right)^{1/(1-\lambda)}>0$$
in (\ref{eq43}) and we compute and reduce, we immediately obtain the desired inequality. Detail is very simple and therefore omitted here.
\end{proof}

As a consequence of the previous theorem, we obtain the following.

\begin{corollary}\label{cor41}
Let $A,B$ and $\lambda$ be as above. Then
\begin{eqnarray}\label{111}
\Big\|\Big(\Re\big(A\sharp_{\lambda}B\big)\Big)^{-1}\Big\|\leq\big\|\big(\Re A\big)^{-1}\big\|^{1-\lambda}\big\|\big(\Re B\big)^{-1}\big\|^\lambda,
\end{eqnarray}
where, for any $T\in{\mathcal B}(H)$, $\|T\|:=\sup_{\|x\|=1}\|Tx\|$ is the usual norm of ${\mathcal B}(H)$.
\end{corollary}
\begin{proof}
Follows from (\ref{eq42}) with $n=1$ and the fact that
$$\|T\|:=\sup_{\|x\|=1}\|Tx\|=\sup_{\|x\|=1}\langle Tx,x\rangle$$
whenever $T\in{\mathcal B}(H)$ is a positive operator.
\end{proof}

It is interesting to see whether \eqref{111} holds for any unitarily invariant norm. Theorem \ref{th41} gives an inequality about $\big(\Re(A\sharp_{\lambda}B)\big)^{-1}$. The following result gives another inequality but involving
$\Re(A\sharp_{\lambda}B)$.

\begin{theorem}\label{th42}
Let $A,B$ and $\lambda$ be as in Theorem \ref{th41}. Then
\begin{equation}\label{eq44}
\Big(\Re e\langle x^*,x\rangle\Big)^2\leq\Big(\langle\Re\big(A\sharp_{\lambda}B\big)x^*,x^*\rangle\Big)\Big(\langle\big(\Re A\big)^{-1}x,x\rangle\sharp_{\lambda}
\langle\big(\Re B\big)^{-1}x,x\rangle\Big)
\end{equation}
for all $x,x^*\in H$
\end{theorem}
\begin{proof}
Following \cite{R}, for any $T\in{\mathcal B}(H)$ strictly positive the following equality
$$\langle T^{-1}x^*,x^*\rangle=\sup_{x\in H}\Big\{2\Re e\langle x^*,x\rangle-\langle Tx,x\rangle\Big\}$$
is valid for all $x^*\in H$. This, with (\ref{eq42}) for $n=1$, immediately implies that
$$2\Re e\langle x^*,x\rangle\leq\langle\Re\big(A\sharp_{\lambda}B\big)x^*,x^*\rangle+\langle\big(\Re A\big)^{-1}x,x\rangle\sharp_{\lambda}
\langle\big(\Re B\big)^{-1}x,x\rangle$$
holds for all $x^*,x\in H$. In this latter inequality we can, of course, replace $x^*$ by $tx^*$ for any real number $t$, for obtaining
\begin{equation}\label{eq45}
2t\Re e\langle x^*,x\rangle\leq t^2\langle\Re\big(A\sharp_{\lambda}B\big)x^*,x^*\rangle+\langle\big(\Re A\big)^{-1}x,x\rangle\sharp_{\lambda}
\langle\big(\Re B\big)^{-1}x,x\rangle.
\end{equation}
If $x^*=0$ the inequality (\ref{eq44}) is obviously an equality. We then assume that $x^*\neq 0$. If in inequality (\ref{eq45}) we take
$$t=\frac{\Re e\langle x^*,x\rangle}{\langle\Re(A\sharp_{\lambda}B)x^*,x^*\rangle}$$
then we obtain, after all reduction, the desired inequality (\ref{eq44}), so completes the proof.
\end{proof}

\bibliographystyle{amsplain}

\end{document}